\documentclass{article}%
\usepackage{amsfonts}
\usepackage{amsmath}
\usepackage{amssymb}
\usepackage{graphicx}%
\providecommand{\U}[1]{\protect\rule{.1in}{.1in}}
\newtheorem{theorem}{Theorem}

\newtheorem{corollary}[theorem]{Corollary}

\newtheorem{proposition}[theorem]{Proposition}
\newtheorem{remark}[theorem]{Remark}

\newenvironment{proof}[1][Proof]{\noindent\textbf{#1.} }{\ \rule{0.5em}{0.5em}}
\begin{document}

\title{Exponential Polynomials and its Applications to the Related Polynomials and Numbers}
\author{Levent Karg\i n\\Alanya Alaaddin Keykubat University \\Akseki Vocational School TR-07630 Antalya Turkey\\leventkargin48@gmail.com}
\maketitle

\begin{abstract}
In this paper we use computational method based on operational point of view
to prove a new generating function for exponential polynomials. We give some
applications involving geometric polynomials, Apostol-Bernoulli and
Apostol-Euler numbers of higher order. We obtain generalized recurrence
relations and explicit expressions for these polynomials and numbers. We also
study some special cases.



\end{abstract}

\section{Introduction, Definitions and Notations}

The exponential polynomials, defined by
\begin{equation}
\phi_{n}\left(  x\right)  =\sum_{k=0}^{n}%
\genfrac{\{}{\}}{0pt}{}{n}{k}%
x^{k}\text{,}\label{1}%
\end{equation}
where $%
\genfrac{\{}{\}}{0pt}{}{n}{k}%
$ is Stirling numbers of the second kind, were first studied by Grunert
\cite{GRUNERT}, with the operational formula%
\[
\left(  xD\right)  ^{n}e^{x}=\phi_{n}\left(  x\right)  e^{x},\text{
\ \ }\left(  xD\right)  =x\frac{d}{dx}.
\]
These polynomials also appear in Ramanujuan's unpublished notebooks. He
obtained the exponential generating function of $\phi_{n}\left(  x\right)  $
as%
\begin{equation}
\sum_{n=0}^{\infty}\phi_{n}\left(  x\right)  \frac{t^{n}}{n!}=e^{x\left(
e^{t}-1\right)  }\label{7}%
\end{equation}
and proved the relation
\begin{equation}
\phi_{n+1}\left(  x\right)  =x\left(  \phi_{n}\left(  x\right)  +\phi
_{n}^{^{\prime}}\left(  x\right)  \right)  .\label{10}%
\end{equation}
We refer to \cite{BL1, Berndt, B, B2, Diletal, MANSOUR, TOUCHARD1, TOUCHARD2}
for details of these polynomials. In particular $\phi_{n}\left(  1\right)  $
are known as exponential numbers (or Bell numbers) $b_{n}$, defined by
(\cite{AS, BL2, C, CG, T})%
\begin{equation}
b_{n}=\phi_{n}\left(  1\right)  =\sum_{k=0}^{n}%
\genfrac{\{}{\}}{0pt}{}{n}{k}%
.\label{6}%
\end{equation}

Geometric polynomials are defined by
\begin{equation}
w_{n}\left(  x\right)  =\sum_{k=0}^{n}%
\genfrac{\{}{\}}{0pt}{}{n}{k}%
k!x^{k}. \label{13}%
\end{equation}
These polynomials are related to the geometric series as follows%
\[
\left(  xD\right)  ^{n}\frac{1}{1-x}=\sum_{k=0}^{\infty}k^{n}x^{k}=\frac
{1}{1-x}w_{n}\left(  \frac{x}{1-x}\right)  ,\text{ \ \ }\left\vert
x\right\vert <1.
\]
We refer to \cite{B, B3, B4, B5, Diletal, DIL2} for details of these
polynomials. By setting $x=1$ in (\ref{13}) one can obtain geometric numbers
(or ordered Bell numbers, or Fubini numbers) $w_{n}$ as%
\begin{equation}
F_{n}=w_{n}\left(  1\right)  =\sum_{k=0}^{n}%
\genfrac{\{}{\}}{0pt}{}{n}{k}%
k!\text{,} \label{14}%
\end{equation}
(\cite{Da, Diletal, Gr}).

Geometric and exponential polynomials are connected by the relation
(\cite{B})
\begin{equation}
w_{n}\left(  x\right)  =\int_{0}^{\infty}\phi_{n}\left(  x\lambda\right)
e^{-\lambda}d\lambda\text{.} \label{16}%
\end{equation}
Thus one can derive the properties of $w_{n}$ from those of $\phi_{n}$. For
example, (\ref{16}) can be used to obtain the exponential generating function
for the geometric polynomials%
\begin{equation}
\frac{1}{1-x\left(  e^{t}-1\right)  }=\sum_{n=0}^{\infty}w_{n}\left(
x\right)  \frac{t^{n}}{n!}\text{.} \label{15}%
\end{equation}

Furthermore as a natural generalization of geometric polynomials, Boyadzhiev
\cite{B} obtain a general class of geometric polynomials as%
\begin{equation}
w_{n,\alpha}\left(  x\right)  =\sum_{k=0}^{n}%
\genfrac{\{}{\}}{0pt}{}{n}{k}%
\binom{\alpha+k-1}{k}k!x^{k}, \label{36}%
\end{equation}
using the operator%
\[
\left(  xD\right)  ^{n}\left\{  \frac{1}{\left(  1-x\right)  ^{\alpha}%
}\right\}  =\frac{1}{\left(  1-x\right)  ^{\alpha}}w_{n,\alpha}\left(
\frac{x}{1-x}\right)  ,\text{ \ \ }\operatorname{Re}\left(  \alpha\right)
>0,\text{ }\left\vert x\right\vert <1.
\]
Therefore, the relation
\[
w_{n}\left(  x\right)  =w_{n,1}\left(  x\right)  ,\text{ \ \ }n\in%
\mathbb{N}
\cup\left\{  0\right\}  ,
\]
follows directly.

The generalized Apostol-Euler polynomials of higher order are defined by means
of the following generating function \cite{LUO2}%
\begin{equation}
\left(  \frac{2}{\lambda e^{t}+1}\right)  ^{\alpha}e^{xt}=\sum_{n=0}^{\infty
}\mathcal{E}_{n}^{\left(  \alpha\right)  }\left(  x;\lambda\right)
\frac{t^{n}}{n!},\text{ \ \ }\left(  \left\vert t\right\vert <\left\vert
\log\left(  -\lambda\right)  \right\vert \right)  , \label{21}%
\end{equation}
Some particular cases of $\mathcal{E}_{n}^{\left(  \alpha\right)  }\left(
x;\lambda\right)  $ are%
\begin{align*}
E_{n}^{\left(  \alpha\right)  }\left(  x\right)   &  =\mathcal{E}_{n}^{\left(
\alpha\right)  }\left(  x;1\right)  \text{ and }E_{n}^{\left(  \alpha\right)
}=E_{n}^{\left(  \alpha\right)  }\left(  0\right)  =\mathcal{E}_{n}^{\left(
\alpha\right)  }\left(  0;1\right)  ,\\
\mathcal{E}_{n}\left(  x;\lambda\right)   &  =\mathcal{E}_{n}^{\left(
1\right)  }\left(  x;\lambda\right)  \text{ and }\mathcal{E}_{n}\left(
\lambda\right)  =\mathcal{E}_{n}^{\left(  1\right)  }\left(  0;\lambda\right)
,\\
E_{n}\left(  x\right)   &  =\mathcal{E}_{n}^{\left(  1\right)  }\left(
x;1\right)  \text{ and }E_{n}=E_{n}\left(  0\right)  =\mathcal{E}_{n}^{\left(
1\right)  }\left(  0;1\right)  ,
\end{align*}
where $E_{n}^{\left(  \alpha\right)  }\left(  x\right)  ,$ $E_{n}^{\left(
\alpha\right)  }$, $\mathcal{E}_{n}\left(  x;\lambda\right)  ,$ $\mathcal{E}%
_{n}\left(  \lambda\right)  ,$ $E_{n}\left(  x\right)  $ and $E_{n}$ denote
Euler polynomials of higher order, Euler numbers of higher order,
Apostol-Euler polynomials, Apostol-Euler numbers, classical Euler polynomials
and classical Euler numbers, respectively. Comparing $\left(  \ref{15}\right)
$ and $\left(  \ref{21}\right)  $ one can obtain%
\begin{equation}
w_{n}\left(  \frac{-1}{2}\right)  =E_{n}. \label{26}%
\end{equation}

The generalized Apostol-Bernoulli polynomials $\mathcal{B}_{n}^{\left(
\alpha\right)  }\left(  x,\lambda\right)  $ of higher order are defined by
means of generating function \cite{SRIandLUO2}%
\begin{equation}
\left(  \frac{t}{\lambda e^{t}-1}\right)  ^{\alpha}e^{xt}=\sum_{n=0}^{\infty
}\mathcal{B}_{n}^{\left(  \alpha\right)  }\left(  x;\lambda\right)
\frac{t^{n}}{n!}, \label{20}%
\end{equation}
where $\left\vert t\right\vert <\pi$ when $\lambda=1$; $\left\vert
t\right\vert <\left\vert \log\lambda\right\vert $ when $\lambda\neq1.$ As
special cases we write%
\begin{align*}
B_{n}^{\left(  \alpha\right)  }\left(  x\right)   &  =\mathcal{B}_{n}^{\left(
\alpha\right)  }\left(  x;1\right)  \text{ and }B_{n}^{\left(  \alpha\right)
}=B_{n}^{\left(  \alpha\right)  }\left(  0\right) \\
\mathcal{B}_{n}\left(  x;\lambda\right)   &  =\mathcal{B}_{n}^{\left(
1\right)  }\left(  x;\lambda\right)  \text{ and }\mathcal{B}_{n}\left(
\lambda\right)  =\mathcal{B}_{n}^{\left(  1\right)  }\left(  0;\lambda\right)
,\\
B_{n}\left(  x\right)   &  =\mathcal{B}_{n}^{\left(  1\right)  }\left(
x;1\right)  \text{ and }B_{n}=B_{n}\left(  0\right)  =\mathcal{B}_{n}^{\left(
1\right)  }\left(  0;1\right)  ,
\end{align*}
where $B_{n}^{\left(  \alpha\right)  }\left(  x\right)  ,$ $B_{n}^{\left(
\alpha\right)  }$, $\mathcal{B}_{n}\left(  x;\lambda\right)  ,$ $\mathcal{B}%
_{n}\left(  \lambda\right)  ,$ $B_{n}\left(  x\right)  $ and $B_{n}$ denote
Bernoulli polynomials of higher order, Bernoulli numbers of higher order,
Apostol-Bernoulli polynomials, Apostol-Bernoulli numbers, classical Bernoulli
polynomials and classical Bernoulli numbers, respectively.

Apostol-Bernoulli numbers $\mathcal{B}_{n}\left(  \lambda\right)  $ have the
following explicit expression \cite[Eq. (3.7)]{APOSTOL}%
\begin{equation}
\mathcal{B}_{n}\left(  \lambda\right)  =\frac{n}{\lambda-1}\sum_{k=0}^{n-1}%
\genfrac{\{}{\}}{0pt}{}{n-1}{k}%
k!\left(  \frac{\lambda}{1-\lambda}\right)  ^{k},\,\ \ \lambda\in%
\mathbb{C}
\backslash\{1\}.\label{53}%
\end{equation}
Thus comparing the above equation with $\left(  \ref{13}\right)  ,$
Apostol-Bernoulli numbers can be expressed by geometric polynomials as
(\cite{B3})%
\begin{equation}
\mathcal{B}_{n}\left(  \lambda\right)  =\frac{n}{\lambda-1}w_{n-1}\left(
\frac{\lambda}{1-\lambda}\right)  ,\text{ \ \ \ }\lambda\in%
\mathbb{C}
\backslash\{1\}.\label{22}%
\end{equation}

We can observe from $\left(  \ref{26}\right)  $ and $\left(  \ref{22}\right)
$ that to study the properties of geometric polynomials is vital for gaining
the known and unknown properties of Bernoulli and Euler numbers which have
important positions in number theory.

The starting point of this study is the formula (\ref{10}), which can be
written as
\[
\phi_{n+1}\left(  x\right)  =\widehat{M}\phi_{n}\left(  x\right)  ,
\]
where $\widehat{M}=\left(  x+xD\right)  .$ The operator $\widehat{M}$ can be
considered as a rising operator acting on the polynomials $\phi_{n}\left(
x\right)  .$ Thus $\phi_{n}\left(  x\right)  $ can be explicitly constructed
the action of $\widehat{M}^{n}$ on $\phi_{0}\left(  x\right)  =1$ as%
\[
\phi_{n}\left(  x\right)  =\widehat{M}^{n}\left\{  1\right\}  .
\]
From this motivation we introduce an operational formula for an appropriate
function $f$ as%
\[
\widehat{M}^{n}f\left(  x\right)  =\sum_{k=0}^{n}\binom{n}{k}\phi_{n-k}\left(
x\right)  \left(  xD\right)  ^{k}f\left(  x\right)  ,
\]
in order to give a direct proof of generalized recurrence relation%
\begin{equation}
\phi_{n+m}\left(  x\right)  =\sum_{k=0}^{n}\sum_{j=0}^{m}\binom{n}{k}%
\genfrac{\{}{\}}{0pt}{}{m}{j}%
j^{n-k}x^{j}\phi_{k}\left(  x\right)  ,\label{11}%
\end{equation}
and to obtain a new generating function for exponential polynomials. Because
of the close relationship with exponential polynomials, we consider the
general geometric polynomials $w_{n,\alpha}\left(  x\right)  ,$ and obtain a
generating function as%
\[
\sum_{n=0}^{\infty}w_{n+m,\alpha}\left(  x\right)  \frac{t^{n}}{n!}=\left(
\frac{1}{1-x\left(  e^{t}-1\right)  }\right)  ^{\alpha}w_{m,\alpha}\left(
\frac{xe^{t}}{1-x\left(  e^{t}-1\right)  }\right)  .
\]
With the help of this generating function we arrive two conclusions. First we
obtain relations between $w_{n,\alpha}\left(  x\right)  $ and Apostol-Euler,
Apostol-Bernoulli numbers of higher order which generalize the equations
$\left(  \ref{26}\right)  $ and $\left(  \ref{22}\right)  .$ Then we get a
generalized recurrence relation for general geometric polynomials
\[
w_{n+m}^{\left(  \alpha\right)  }\left(  x\right)  =\sum_{k=0}^{m}\sum
_{j=0}^{n}%
\genfrac{\{}{\}}{0pt}{}{m}{k}%
\binom{n}{j}\binom{\alpha+k-1}{k}k!k^{n-j}x^{k}w_{j}^{\left(  \alpha+k\right)
}\left(  x\right)  .
\]
Dealing with recurrence relations of Apostol-Euler and Apostol-Bernoulli
numbers of higher order, which are in particular useful in computing special
values of these numbers is considerably important and studied by many
mathematicians (see \cite{MUMIN, CENKCI, HOWARD1, LUO2, LUO3, SRIandLUO2,
SRIandLUO3, SRIVASTAVA3, TODOROV}). Therefore, the connections between
$w_{n,\alpha}\left(  x\right)  $ and these numbers enable us to obtain
generalized recurrence relations for Apostol-Euler and Apostol-Bernoulli
numbers of higher order. We also give new generating functions of these
numbers. Furthermore we deduce their special cases involving Euler and
Bernoulli numbers of higher order, Apostol-Euler and Apostol-Bernoulli numbers
and classical Euler and Bernoulli numbers.

The organization of this paper is as follows. In section 2 we examine
$\phi_{n+m}\left(  x\right)  $ for $n,m\in%
\mathbb{N}
\cup\left\{  0\right\}  $. We give a different proof for $\left(
\ref{11}\right)  $ and obtain a new generating function. In section 3 we deal
with general geometric polynomials, Apostol-Euler and Apostol-Bernoulli
numbers of higher order and find new generating functions and recurrence relations.

Now we state our results.

\section{An Operational Formula and Its Results for Exponential Polynomials}

In this section we obtain an operational formula involving exponential
polynomials and a new generating function of these polynomials. First we give
the following theorem.

\begin{theorem}
The following operational formula holds for an appropriate function $f$:%
\begin{equation}
\widehat{M}^{n}f\left(  x\right)  =\sum_{k=0}^{n}\binom{n}{k}\phi_{n-k}\left(
x\right)  \left(  xD\right)  ^{k}f\left(  x\right)  , \label{3}%
\end{equation}
where $\widehat{M}=\left(  x+xD\right)  .$
\end{theorem}

\begin{proof}
Proof follows from induction on $n.$ Assume the statement holds for $n$; that
is,%
\[
\widehat{M}^{n}f\left(  x\right)  =\sum_{k=0}^{n}\binom{n}{k}\phi_{n-k}\left(
x\right)  \left(  xD\right)  ^{k}f\left(  x\right)  .
\]
We need to show that this is true for $n+1$.%
\begin{align*}
&  \widehat{M}^{n+1}f\left(  x\right)  =\widehat{M}\widehat{M}^{n}f\left(
x\right) \\
&  \quad=\sum_{k=0}^{n}\binom{n}{k}x\phi_{n-k}\left(  x\right)  \left(
xD\right)  ^{k}f\left(  x\right)  +\sum_{k=0}^{n}\binom{n}{k}\left(
xD\right)  \left[  \phi_{n-k}\left(  x\right)  \left(  xD\right)  ^{k}f\left(
x\right)  \right]  .
\end{align*}
Using $\left(  \ref{10}\right)  $ we have
\begin{align*}
&  \widehat{M}^{n+1}f\left(  x\right) \\
&  \quad=\sum_{k=0}^{n}\binom{n}{k}\phi_{n+1-k}\left(  x\right)  \left(
xD\right)  ^{k}f\left(  x\right)  +\sum_{k=1}^{n+1}\binom{n}{k-1}\phi
_{n+1-k}\left(  x\right)  \left(  xD\right)  ^{k}f\left(  x\right) \\
&  \quad=\phi_{n+1}\left(  x\right)  +\left(  xD\right)  ^{n+1}f\left(
x\right)  +\sum_{k=1}^{n}\left[  \binom{n}{k}+\binom{n}{k-1}\right]
\phi_{n+1-k}\left(  x\right)  \left(  xD\right)  ^{k}f\left(  x\right)  .
\end{align*}
Finally using the well-known binomial recursion relation%
\[
\binom{n+1}{k}=\binom{n}{k}+\binom{n}{k-1}%
\]
gives the statement is true for $n+1$.
\end{proof}

Setting $f\left(  x\right)  =\phi_{m}\left(  x\right)  $ in (\ref{3}) and
using the definition of exponential polynomials given by (\ref{1}) we get the
following dentity which also occurs in \cite{BELHACIR, B2, GOULD}.

\begin{corollary}
For $n,m\in%
\mathbb{N}
\cup\left\{  0\right\}  ,$%
\begin{equation}
\phi_{n+m}\left(  x\right)  =\sum_{k=0}^{n}\sum_{j=0}^{m}\binom{n}{k}%
\genfrac{\{}{\}}{0pt}{}{m}{j}%
j^{n-k}x^{j}\phi_{k}\left(  x\right)  .\label{27}%
\end{equation}

\end{corollary}

We note that setting $x=1$, $\left(  \ref{27}\right)  $ coincides with the
formula which was given by Spivey (\cite{SPIVEY}) for the Bell numbers.

A natural question is to find the generating function for $\phi_{n+m}\left(
x\right)  .$ To answer this we need the following theorem.

\begin{theorem}
Let
\[
g\left(  x\right)  =\sum_{k=0}^{\infty}c_{k}x^{k}%
\]
be the formal series expansion of $g\left(  x\right)  $. Then we have
\begin{equation}
e^{t\widehat{M}}g\left(  x\right)  =e^{x\left(  e^{t}-1\right)  }g\left(
xe^{t}\right)  . \label{5}%
\end{equation}

\end{theorem}

\begin{proof}
Acting the operator to $g\left(  x\right)  $ we get%
\begin{equation}
e^{t\widehat{M}}g\left(  x\right)  =\sum_{n=0}^{\infty}\sum_{k=0}^{\infty
}\frac{t^{n}}{n!}\widehat{M}^{n}c_{k}x^{k}. \label{4}%
\end{equation}
For $f\left(  x\right)  =x^{k},$ (\ref{3}) and (\ref{4}) show that%
\begin{align*}
e^{t\widehat{M}}g\left(  x\right)   &  =\left[  \sum_{n=0}^{\infty}\phi
_{n}\left(  x\right)  \frac{t^{n}}{n!}\right]  \left[  \sum_{k=0}^{\infty
}c_{k}\left(  \sum_{m=0}^{\infty}\frac{k^{m}t^{m}}{m!}\right)  x^{k}\right] \\
&  =e^{x\left(  e^{t}-1\right)  }g\left(  xe^{t}\right)  .
\end{align*}

\end{proof}

For a use of this operational formula, we set $g\left(  x\right)  =\phi
_{m}\left(  x\right)  $ in $\left(  \ref{5}\right)  .$ Then we obtain the
following theorem which gives a new generating function for the exponential polynomials.

\begin{corollary}
For $m\in%
\mathbb{N}
\cup\left\{  0\right\}  ,$ the exponential polynomials have the following
generating function%
\begin{equation}
\sum_{n=0}^{\infty}\phi_{n+m}\left(  x\right)  \frac{t^{n}}{n!}=e^{x\left(
e^{t}-1\right)  }\phi_{m}\left(  xe^{t}\right)  ,\text{ \ \ }t\in%
\mathbb{C}
. \label{8}%
\end{equation}

\end{corollary}

Since $\phi_{0}\left(  x\right)  =1,$ $\left(  \ref{8}\right)  $ reduce to
$\left(  \ref{7}\right)  ,$ obtained by Ramanujuan. Moreover setting $x=1$ and
$t=ki\pi$ $\left(  k\in%
\mathbb{Z}
\right)  $ in $\left(  \ref{8}\right)  $ we get the value of infinite
summation of Bell numbers
\[
\sum_{n=0}^{\infty}b_{n+m}\frac{\left(  ki\pi\right)  ^{n}}{n!}=\left\{
\begin{array}
[c]{cc}%
b_{m} & ,\text{ if }k\text{ even}\\
e^{-2}\tilde{b}_{m} & ,\text{ if }k\text{ odd}%
\end{array}
\right.
\]
where $\tilde{b}_{m}$ is the complementary Bell numbers (or Uppuluri-Carpenter
numbers) (\cite{BEARD, Berndt, Uppuluri}).

\section{Applications to the Related Polynomials and Numbers}

In this section we give some applications of $\left(  \ref{8}\right)  $ for
several polynomials and numbers, which are closely related to $\phi_{n}\left(
x\right)  .$

Formula $\left(  \ref{16}\right)  $ can be extended to the integral
representation%
\begin{equation}
w_{n,\alpha}\left(  x\right)  =\frac{1}{\Gamma\left(  \alpha\right)  }\int%
_{0}^{\infty}\lambda^{\alpha-1}\phi_{n}\left(  x\lambda\right)  e^{-\lambda
}d\lambda, \label{37}%
\end{equation}
which is verified immediately by using $\left(  \ref{1}\right)  $, $\left(
\ref{36}\right)  $ and Euler's integral representation for the gamma
function$.$ Thus the exponential generating function for general geometric
polynomials can be found by writing $\left(  \ref{8}\right)  $ in the form
\[
\sum_{n=0}^{\infty}\phi_{n+m}\left(  x\lambda\right)  \frac{t^{n}}%
{n!}=e^{x\lambda\left(  e^{t}-1\right)  }\phi_{m}\left(  x\lambda
e^{t}\right)  ,
\]
then multiplying both sides by $\lambda^{\alpha-1}e^{-\lambda}$ and
integrating for $\lambda$ from zero to infinity. In the view of $\left(
\ref{37}\right)  $ this gives
\[
\sum_{n=0}^{\infty}w_{n+m,\alpha}\left(  x\right)  \frac{t^{n}}{n!}=\frac
{1}{\Gamma\left(  \alpha\right)  }\int_{0}^{\infty}\lambda^{\alpha
-1}e^{-\lambda\left(  1-x\left(  e^{t}-1\right)  \right)  }\phi_{m}\left(
x\lambda e^{t}\right)  d\lambda,
\]
and this equation leads the following result by a simple change of variable.

\begin{theorem}
\label{teo2}We have the following generating function:%
\[
\sum_{n=0}^{\infty}w_{n+m,\alpha}\left(  x\right)  \frac{t^{n}}{n!}=\left(
\frac{1}{1-x\left(  e^{t}-1\right)  }\right)  ^{\alpha}w_{m,\alpha}\left(
\frac{xe^{t}}{1-x\left(  e^{t}-1\right)  }\right)  ,
\]
where $m\in%
\mathbb{N}
\cup\left\{  0\right\}  $ and $\alpha\in%
\mathbb{C}
$ with $\operatorname{Re}\left(  \alpha\right)  >0.$
\end{theorem}

For $\alpha=1$ in the above theorem we have the following generating function
for $w_{n}\left(  x\right)  $.

\begin{corollary}
We have
\begin{equation}
\sum_{n=0}^{\infty}w_{n+m}\left(  x\right)  \frac{t^{n}}{n!}=\frac
{1}{1-x\left(  e^{t}-1\right)  }w_{m}\left(  \frac{xe^{t}}{1-x\left(
e^{t}-1\right)  }\right)  . \label{12}%
\end{equation}

\end{corollary}

Since $w_{0}\left(  x\right)  =1$ $\left(  \ref{12}\right)  $ reduce to
$\left(  \ref{15}\right)  $ obtained by Boyadzhiev. Also, for $m=0,$ we have
$w_{0,\alpha}\left(  x\right)  =1,$ which leads the following generating
function for the general geometric polynomials as
\begin{equation}
\sum_{n=0}^{\infty}w_{n,\alpha}\left(  x\right)  \frac{t^{n}}{n!}=\left(
\frac{1}{1-x\left(  e^{t}-1\right)  }\right)  ^{\alpha},\label{38}%
\end{equation}
which is also given in (\cite{DILANDB1}) by different way.

From now on we use the notation $w_{n}^{\left(  \alpha\right)  }\left(
x\right)  $ instead of $w_{n,\alpha}\left(  x\right)  $ since $\left(
\ref{38}\right)  $ has some similarities with the generating functions of to
the generalized Apostol-Bernoulli and Apostol-Euler numbers of higher order,
for example
\begin{equation}
w_{n,\alpha}\left(  \frac{-\lambda}{\lambda+1}\right)  =\left(  \frac
{\lambda+1}{2}\right)  ^{\alpha}\mathcal{E}_{n}^{\left(  \alpha\right)
}\left(  \lambda\right)  ,\text{ \ \ }\lambda\in%
\mathbb{C}
\label{40}%
\end{equation}
and%
\begin{equation}
w_{n,l}\left(  \frac{-\lambda}{\lambda-1}\right)  =\frac{\left(
\lambda-1\right)  ^{l}}{l!}\binom{n+l}{l}^{-1}\mathcal{B}_{n+l}^{\left(
l\right)  }\left(  \lambda\right)  ,\text{ \ }\lambda\in%
\mathbb{C}
\backslash\{1\}\text{ and }\alpha=l\in%
\mathbb{N}
. \label{47}%
\end{equation}
Therefore one can derive properties of generalized Apostol-Bernoulli and
Apostol-Euler numbers of higher order from general geometric polynomials.
Moreover setting $\alpha=\lambda=1$ in $\left(  \ref{40}\right)  $ and $l=1$
in $\left(  \ref{47}\right)  $ we have $\left(  \ref{26}\right)  $ and
$\left(  \ref{22}\right)  $ respectively.

Now we state some applications of Theorem \ref{teo2}. First we have the
following corollary.

\begin{corollary}
We have the generating functions for Apostol-Euler and Apostol-Bernoulli
numbers of higher order as:%
\begin{align}
\sum_{n=0}^{\infty}\mathcal{E}_{n+m}^{\left(  \alpha\right)  }\left(
\lambda\right)  \frac{t^{n}}{n!}  &  =\left(  \frac{2}{\lambda e^{t}%
+1}\right)  ^{\alpha}w_{m}^{\left(  \alpha\right)  }\left(  \frac{-\lambda
e^{t}}{\lambda e^{t}+1}\right)  ,\label{48}\\
\sum_{n=0}^{\infty}\binom{n+m+l}{l}^{-1}\mathcal{B}_{n+m+l}^{\left(  l\right)
}\left(  \lambda\right)  \frac{t^{n}}{n!}  &  =l!\left(  \frac{1}{\lambda
e^{t}-1}\right)  ^{l}w_{m}^{\left(  l\right)  }\left(  \frac{-\lambda e^{t}%
}{\lambda e^{t}-1}\right)  . \label{43}%
\end{align}

\end{corollary}

Setting $\lambda=1$, $\alpha=1$ and $\lambda=\alpha=1$ in $\left(
\ref{48}\right)  $ we have the following generating functions for Euler
numbers of higher order, Apostol-Euler numbers and classical Euler numbers%

\begin{align*}
\sum_{n=0}^{\infty}E_{n+m}^{\left(  \alpha\right)  }\frac{t^{n}}{n!}  &
=\left(  \frac{2}{e^{t}+1}\right)  ^{\alpha}w_{m}^{\left(  \alpha\right)
}\left(  \frac{-e^{t}}{e^{t}+1}\right)  ,\\
\sum_{n=0}^{\infty}\mathcal{E}_{n+m}\left(  \lambda\right)  \frac{t^{n}}{n!}
&  =\frac{2}{\lambda e^{t}+1}w_{m}\left(  \frac{-\lambda e^{t}}{\lambda
e^{t}+1}\right)  ,\\
\sum_{n=0}^{\infty}E_{n+m}\frac{t^{n}}{n!}  &  =\frac{2}{e^{t}+1}w_{m}\left(
\frac{-e^{t}}{e^{t}+1}\right)  ,
\end{align*}
respectively. Similarly, setting appropriate parameters in $\left(
\ref{43}\right)  $ we have the following results for Bernoulli numbers of
higher order, Apostol-Bernoulli numbers and classical Bernoulli numbers%
\begin{align*}
\sum_{n=0}^{\infty}\binom{n+m+l}{l}^{-1}B_{n+m+l}^{\left(  l\right)  }%
\frac{t^{n}}{n!}  &  =l!\left(  \frac{1}{e^{t}-1}\right)  ^{l}w_{m}^{\left(
l\right)  }\left(  \frac{-e^{t}}{e^{t}-1}\right) \\
\sum_{n=0}^{\infty}\frac{\mathcal{B}_{n+m+1}\left(  \lambda\right)  }%
{n+m+1}\frac{t^{n}}{n!}  &  =\frac{1}{\lambda e^{t}-1}w_{m}\left(
\frac{-\lambda e^{t}}{\lambda e^{t}-1}\right) \\
\sum_{n=0}^{\infty}\frac{B_{n+m+1}}{n+m+1}\frac{t^{n}}{n!}  &  =\frac{1}%
{e^{t}-1}w_{m}\left(  \frac{e^{t}}{1-e^{t}}\right)  ,
\end{align*}
respectively. Finally, for $t=i\pi$ in $\left(  \ref{48}\right)  \ $and
$\left(  \ref{43}\right)  $ we have the values of infinite summation for
Apostol-Euler and Apostol-Bernoulli numbers of higher order%
\begin{align*}
\sum_{n=0}^{\infty}\mathcal{E}_{n+m}^{\left(  l\right)  }\left(
\lambda\right)  \frac{\left(  i\pi\right)  ^{n}}{n!}  &  =\frac{\left(
-2\right)  ^{l}}{l!}\binom{m+l}{l}^{-1}\mathcal{B}_{m+l}^{\left(  l\right)
}\left(  \lambda\right)  ,\\
\sum_{n=0}^{\infty}\binom{n+m+l}{l}^{-1}\mathcal{B}_{n+m+l}^{\left(  l\right)
}\left(  \lambda\right)  \frac{\left(  i\pi\right)  ^{n}}{n!}  &  =\left(
\frac{-1}{2}\right)  ^{l}l!\mathcal{E}_{m}^{\left(  l\right)  }\left(
\lambda\right)  .
\end{align*}
Similarly for $t=2i\pi$ we have
\begin{align*}
\sum_{n=0}^{\infty}\mathcal{E}_{n+m}^{\left(  \alpha\right)  }\left(
\lambda\right)  \frac{\left(  2i\pi\right)  ^{n}}{n!}  &  =\mathcal{E}%
_{m}^{\left(  \alpha\right)  }\left(  \lambda\right)  ,\\
\sum_{n=1}^{\infty}\binom{n+m+l}{l}^{-1}\mathcal{B}_{n+m+l}^{\left(  l\right)
}\left(  \lambda\right)  \frac{\left(  2i\pi\right)  ^{n}}{n!}  &  =0.
\end{align*}

The second application of Theorem \ref{teo2} is given in the following theorem.

\begin{theorem}
\label{teo1}We have the following generalized recurrence relation for general
geometric polynomials%
\begin{equation}
w_{n+m}^{\left(  \alpha\right)  }\left(  x\right)  =\sum_{k=0}^{m}\sum
_{j=0}^{n}%
\genfrac{\{}{\}}{0pt}{}{m}{k}%
\binom{n}{j}\binom{\alpha+k-1}{k}k!k^{n-j}x^{k}w_{j}^{\left(  \alpha+k\right)
}\left(  x\right)  , \label{46}%
\end{equation}
where $n,m\in%
\mathbb{N}
\cup\left\{  0\right\}  $.
\end{theorem}

\begin{proof}
Using $\left(  \ref{36}\right)  $ in $\left(  \ref{38}\right)  ,$ we have
\begin{align*}
\sum_{n=0}^{\infty}w_{n+m}^{\left(  \alpha\right)  }\left(  x\right)
\frac{t^{n}}{n!}  &  =\sum_{k=0}^{m}%
\genfrac{\{}{\}}{0pt}{}{m}{k}%
\binom{\alpha+k-1}{k}k!x^{k}\left(  \frac{1}{1-x\left(  e^{t}-1\right)
}\right)  ^{\alpha+k}e^{kt}\\
&  =\sum_{k=0}^{m}%
\genfrac{\{}{\}}{0pt}{}{m}{k}%
\binom{\alpha+k-1}{k}k!x^{k}\left[  \sum_{j=0}^{\infty}w_{j}^{\left(
\alpha+k\right)  }\left(  x\right)  \frac{t^{j}}{j!}\right]  \left[
\sum_{n=0}^{\infty}\frac{k^{n}t^{n}}{n!}\right] \\
&  =\sum_{n=0}^{\infty}\left[  \sum_{k=0}^{m}\sum_{j=0}^{n}%
\genfrac{\{}{\}}{0pt}{}{m}{k}%
\binom{n}{j}\binom{\alpha+k-1}{k}k!k^{n-j}x^{k}w_{j}^{\left(  \alpha+k\right)
}\left(  x\right)  \right]  \frac{t^{n}}{n!}.
\end{align*}
Comparing the coefficients of $\frac{t^{n}}{n!}$ gives the desired equation.
\end{proof}

We note that Equation (\ref{46}) is also given in (\cite{DILANDB1}) however
the proofs are different. Moreover for $n=0$, $\left(  \ref{46}\right)  $
reduces to $\left(  \ref{36}\right)  $. Other results which can be drawn from
Theorem \ref{teo1} can be listed as follows. First we deal with geometric
polynomials. Setting $\alpha=1$ in $\left(  \ref{46}\right)  $ we have
\[
w_{n+m}\left(  x\right)  =\sum_{k=0}^{m}\sum_{j=0}^{n}%
\genfrac{\{}{\}}{0pt}{}{m}{k}%
\binom{n}{j}k!k^{n-j}x^{k}w_{j}^{\left(  k+1\right)  }\left(  x\right)  .
\]
Then using $\left(  \ref{36}\right)  $ in the above equation we get the
following corollary.

\begin{corollary}
The following explicit expression holds for geometric polynomials%
\begin{equation}
w_{n+m}\left(  x\right)  =\sum_{k=0}^{m}\sum_{j=0}^{n}\sum_{i=0}^{j}%
\genfrac{\{}{\}}{0pt}{}{m}{k}%
\binom{n}{j}%
\genfrac{\{}{\}}{0pt}{}{j}{i}%
k^{n-j}\left(  i+k\right)  !x^{k}, \label{41}%
\end{equation}
where $n,m\in%
\mathbb{N}
\cup\left\{  0\right\}  .$
\end{corollary}

Setting $x=1$ in $\left(  \ref{41}\right)  $ yields an explicit expression for
the geometric numbers as%
\begin{equation}
F_{n+m}=\sum_{k=0}^{m}\sum_{j=0}^{n}\sum_{i=0}^{j}%
\genfrac{\{}{\}}{0pt}{}{m}{k}%
\binom{n}{j}%
\genfrac{\{}{\}}{0pt}{}{j}{i}%
k^{n-j}\left(  i+k\right)  !. \label{18}%
\end{equation}
For $n=0$ or $m=0,$ $\left(  \ref{18}\right)  $ reduces to $\left(
\ref{13}\right)  .$

Now we mention the results for generalized Apostol-Euler numbers of higher
order. Setting $x=\frac{-\lambda}{\lambda+1}$ in $\left(  \ref{46}\right)  $
and using $\left(  \ref{40}\right)  $ we have the following corollary.

\begin{corollary}
The following generalized recurrence relation holds for Apostol-Euler numbers
of higher order%
\begin{equation}
\mathcal{E}_{n+m}^{\left(  \alpha\right)  }\left(  \lambda\right)  =\sum
_{k=0}^{m}\sum_{j=0}^{n}\binom{n}{j}%
\genfrac{\{}{\}}{0pt}{}{m}{k}%
\binom{\alpha+k-1}{k}\frac{\left(  -\lambda\right)  ^{k}k!k^{n-j}}{2^{k}%
}\mathcal{E}_{j}^{\left(  \alpha+k\right)  }\left(  \lambda\right)  .
\label{52}%
\end{equation}

\end{corollary}

For the natural consequences of the above corollary, setting $\lambda=1$,
$\alpha=1$ and $\lambda=\alpha=1$ in $\left(  \ref{52}\right)  $ we obtain the
following generalized recurrence relations for Euler numbers of higher order,
Apostol-Euler numbers and classical Euler numbers%
\begin{align}
E_{n+m}^{\left(  \alpha\right)  }  &  =\sum_{k=0}^{m}\sum_{j=0}^{n}%
\genfrac{\{}{\}}{0pt}{}{m}{k}%
\binom{n}{j}\binom{\alpha+k-1}{k}\frac{k!k^{n-j}}{\left(  -2\right)  ^{k}%
}E_{j}^{\left(  \alpha+k\right)  },\label{55}\\
\mathcal{E}_{n+m}\left(  \lambda\right)   &  =\sum_{k=0}^{m}\sum_{j=0}^{n}%
\genfrac{\{}{\}}{0pt}{}{m}{k}%
\binom{n}{j}\frac{\left(  -\lambda\right)  ^{k}k^{n-j}k!}{2^{k}}%
\mathcal{E}_{j}^{\left(  k+1\right)  }\left(  \lambda\right)  ,\label{58}\\
E_{n+m}  &  =\sum_{k=0}^{m}\sum_{j=0}^{n}%
\genfrac{\{}{\}}{0pt}{}{m}{k}%
\binom{n}{j}\frac{\left(  -1\right)  ^{k}k!}{2^{k}}E_{j}^{\left(  k+1\right)
}, \label{59}%
\end{align}
respectively.

\begin{remark}
\label{re1}If $n=0$ in $\left(  \ref{52}\right)  $ and if we use
$\mathcal{E}_{0}^{\left(  \alpha\right)  }\left(  \lambda\right)  =2^{\alpha
}\left(  \lambda+1\right)  ^{-\alpha},$ we get the following explicit
expression for Apostol-Euler numbers of higher order
\begin{equation}
\mathcal{E}_{m}^{\left(  \alpha\right)  }\left(  \lambda\right)  =2^{\alpha
}\sum_{k=0}^{m}%
\genfrac{\{}{\}}{0pt}{}{m}{k}%
\binom{\alpha+k-1}{k}\frac{k!\left(  -\lambda\right)  ^{k}}{\left(
\lambda+1\right)  ^{\alpha+k}}. \label{42}%
\end{equation}
Furthermore, for $\lambda=1$, $\alpha=1$ and $\lambda=\alpha=1$ in $\left(
\ref{42}\right)  $ we obtain explicit expressions for Euler numbers of higher
order$,$ Apostol-Euler numbers and classical Euler numbers as%
\begin{align}
E_{m}^{\left(  \alpha\right)  }  &  =\sum_{k=0}^{m}%
\genfrac{\{}{\}}{0pt}{}{m}{k}%
\binom{\alpha+k-1}{k}\frac{\left(  -1\right)  ^{k}k!}{2^{k}},\label{66}\\
\mathcal{E}_{m}\left(  \lambda\right)   &  =2\sum_{k=0}^{m}%
\genfrac{\{}{\}}{0pt}{}{m}{k}%
\frac{\left(  -\lambda\right)  ^{k}k!}{\left(  \lambda+1\right)  ^{k+1}%
},\label{67}\\
E_{m}  &  =\sum_{k=0}^{m}%
\genfrac{\{}{\}}{0pt}{}{m}{k}%
\frac{\left(  -1\right)  ^{k}k!}{2^{k}}, \label{68}%
\end{align}
respectively. We refer $\cite{LUO2}$ for details. Besides, $\left(
\ref{68}\right)  $ is also given by $\cite{B4}$ by different mean.
\end{remark}

Finally we mention the results of Theorem \ref{teo1} for generalized
Apostol-Bernoulli numbers of higher order. Setting $x=\frac{-\lambda}%
{\lambda-1}$ in $\left(  \ref{46}\right)  $ and using $\left(  \ref{47}%
\right)  $ we have the following corollary.

\begin{corollary}
The following generalized recurrence relation holds for Apostol-Bernoulli
numbers of higher order:%
\begin{equation}
\frac{\mathcal{B}_{n+m+l}^{\left(  l\right)  }\left(  \lambda\right)  }%
{\binom{n+m+l}{l}l}=\sum_{k=0}^{m}\sum_{j=0}^{n}%
\genfrac{\{}{\}}{0pt}{}{m}{k}%
\binom{n}{j}\binom{l+k+j}{j}^{-1}\frac{\left(  -\lambda\right)  ^{k}k^{n-j}%
}{l+k}\mathcal{B}_{l+k+j}^{\left(  l+k\right)  }\left(  \lambda\right)  .
\label{49}%
\end{equation}

\end{corollary}

Setting $\lambda=1$, $l=1$ and $\lambda=l=1$ in $\left(  \ref{49}\right)  $ we
obtain generalized recurrence relations for Bernoulli numbers of higher order,
Apostol-Bernoulli numbers and classical Bernoulli numbers%
\begin{align}
\frac{B_{n+m+l}^{\left(  l\right)  }}{\binom{n+m+l}{l}l}  &  =\sum_{k=0}%
^{m}\sum_{j=0}^{n}%
\genfrac{\{}{\}}{0pt}{}{m}{k}%
\binom{n}{j}\binom{l+k+j}{j}^{-1}\frac{\left(  -1\right)  ^{k}k^{n-j}}%
{l+k}B_{l+k+j}^{\left(  l+k\right)  },\label{56}\\
\frac{\mathcal{B}_{n+m+1}\left(  \lambda\right)  }{\left(  n+m+1\right)  }  &
=\sum_{k=0}^{m}\sum_{j=0}^{n}%
\genfrac{\{}{\}}{0pt}{}{m}{k}%
\binom{n}{j}\binom{k+j+1}{j}^{-1}\frac{\left(  -\lambda\right)  ^{k}k^{n-j}%
}{k+1}\mathcal{B}_{j+k+1}^{\left(  k+1\right)  }\left(  \lambda\right)
,\label{57}\\
\frac{B_{n+m+1}}{\left(  n+m+1\right)  }  &  =\sum_{k=0}^{m}\sum_{j=0}^{n}%
\genfrac{\{}{\}}{0pt}{}{m}{k}%
\binom{n}{j}\binom{k+j+1}{j}^{-1}\frac{\left(  -1\right)  ^{k}k^{n-j}}%
{k+1}B_{j+k+1}^{\left(  k+1\right)  }, \label{17}%
\end{align}
respectively. Besides, for $n=0$ in $\left(  \ref{56}\right)  $ we have a
recurrence relation for Bernoulli numbers of higher order involving Bernoulli
numbers of higher order in the following corollary.

\begin{corollary}
Bernoulli numbers of higher order hold
\begin{equation}
B_{m+l}^{\left(  l\right)  }=l\binom{m+l}{l}\sum_{k=0}^{m}%
\genfrac{\{}{\}}{0pt}{}{m}{k}%
\frac{\left(  -1\right)  ^{k}k!}{\left(  l+k\right)  }B_{l+k}^{\left(
l+k\right)  }. \label{65}%
\end{equation}

\end{corollary}

Applying Stirling transform $\left(  \text{see }\cite{RIORDAN}\text{ for
details}\right)  $ to $\left(  \ref{11}\right)  $ we have another recurrence
relation as%
\begin{equation}
B_{m+l}^{\left(  m+l\right)  }=\frac{\left(  l+m\right)  }{m!l}\sum_{k=0}%
^{m}\left(  -1\right)  ^{k}%
\genfrac{[}{]}{0pt}{}{m}{k}%
\binom{k+l}{l}^{-1}B_{k+l}^{\left(  l\right)  }. \label{39}%
\end{equation}
It is good to note that for $l=1$ in $\left(  \ref{39}\right)  $ we get%
\begin{equation}
B_{m+1}^{\left(  m+1\right)  }=\frac{m+1}{m!}\sum_{k=0}^{m}\left(  -1\right)
^{k}%
\genfrac{[}{]}{0pt}{}{m}{k}%
\frac{B_{k+1}}{k+1}, \label{44}%
\end{equation}
which is the special case of the formula given by N\"{o}rlund in
\cite{NORLUND}.

\begin{remark}
\label{re2}For $n=0$ in $\left(  \ref{49}\right)  $ we obtain a recurrence
relation for Apostol-Bernoulli numbers of higher order%
\begin{equation}
\mathcal{B}_{m+l}^{\left(  l\right)  }\left(  \lambda\right)  =l\binom{m+l}%
{l}\sum_{k=0}^{m}%
\genfrac{\{}{\}}{0pt}{}{m}{k}%
\frac{\left(  -\lambda\right)  ^{k}k!}{\left(  l+k\right)  }\mathcal{B}%
_{l+k}^{\left(  l+k\right)  }\left(  \lambda\right)  .\label{50}%
\end{equation}
Replacing $m+l$ with $n$ and setting $\mathcal{B}_{k}^{\left(  k\right)
}\left(  \lambda\right)  =\frac{k!}{\left(  \lambda-1\right)  ^{k}}$ in
$\left(  \ref{50}\right)  $ we obtain the explicit representation
\begin{equation}
\mathcal{B}_{n}^{\left(  l\right)  }\left(  \lambda\right)  =l!\binom{n}%
{l}\sum_{k=0}^{n-l}%
\genfrac{\{}{\}}{0pt}{}{n-l}{k}%
\binom{l+k-1}{k}\frac{\left(  -\lambda\right)  ^{k}k!}{\left(  \lambda
-1\right)  ^{l+k}},\label{54}%
\end{equation}
which is given by Srivastava et al in \cite[Eq. (27)]{SRIandLUO2}. Moreover
Apostol's formula (\ref{53}) is a special case of the formula (\ref{57}) when
$n=0$ and $l=1.$
\end{remark}

As we mentioned before, calculating the values of the Apostol-Euler and
Apostol-Bernoulli polynomials of higher order is an active working area. In
order to give effective calculation formulas for these polynomials we first
need the following proposition.

\begin{proposition}
\label{pro1}We have the following generalized recurrence relations:%
\begin{align}
\mathcal{E}_{n+m}^{\left(  \alpha\right)  }\left(  \lambda\right)   &
=\sum_{k=0}^{m}%
\genfrac{\{}{\}}{0pt}{}{m}{k}%
\binom{\alpha+k-1}{k}\left(  \frac{-\lambda}{2}\right)  ^{k}k!\mathcal{E}%
_{n}^{\left(  \alpha+k\right)  }\left(  k;\lambda\right)  ,\label{19}\\
\frac{\mathcal{B}_{n+m+l}^{\left(  l\right)  }\left(  \lambda\right)  }%
{\binom{n+m+l}{l}}  &  =\sum_{k=0}^{m}%
\genfrac{\{}{\}}{0pt}{}{m}{k}%
\binom{n+l+k}{n}^{-1}\frac{l\left(  -\lambda\right)  ^{k}}{\left(  l+k\right)
}\mathcal{B}_{n+l+k}^{\left(  k+l\right)  }\left(  k;\lambda\right)  .
\label{25}%
\end{align}

\end{proposition}

\begin{proof}
From $\left(  \ref{36}\right)  $ we can write $\left(  \ref{48}\right)  $ as
\[
\sum_{n=0}^{\infty}\mathcal{E}_{n+m}^{\left(  \alpha\right)  }\left(
\lambda\right)  \frac{t^{n}}{n!}=\sum_{k=0}^{m}%
\genfrac{\{}{\}}{0pt}{}{m}{k}%
\binom{\alpha+k-1}{k}\frac{k!\left(  -\lambda\right)  ^{k}}{2^{k}}\left(
\frac{2}{\lambda e^{t}+1}\right)  ^{\alpha+k}e^{kt}.
\]
Using $\left(  \ref{21}\right)  $ in the last part of the above equation we
get
\[
\sum_{n=0}^{\infty}\mathcal{E}_{n+m}^{\left(  \alpha\right)  }\left(
\lambda\right)  \frac{t^{n}}{n!}=\sum_{k=0}^{m}%
\genfrac{\{}{\}}{0pt}{}{m}{k}%
\binom{\alpha+k-1}{k}\frac{k!\left(  -\lambda\right)  ^{k}}{2^{k}}\sum
_{n=0}^{\infty}\mathcal{E}_{n}^{\left(  \alpha+k\right)  }\left(
k;\lambda\right)  \frac{t^{n}}{n!}.
\]
Comparing the coefficients of $\frac{t^{n}}{n!}$ we get the desired equation
$\left(  \ref{19}\right)  $.

Equation $\left(  \ref{25}\right)  $ can be proved the same method by using
$\left(  \ref{20}\right)  .$
\end{proof}

Equations $\left(  \ref{42}\right)  $ and $\left(  \ref{50}\right)  $ are the
special case $n=0$ of the Proposition \ref{pro1}. Now it is better to improve
the equations $\left(  \ref{19}\right)  $ and $\left(  \ref{25}\right)  $ as
generalized recurrence relations for these polynomials by induction in the
following theorem.

\begin{theorem}
\label{teo3}We have the following generalized recurrence relations for
Apostol-Euler and Apostol-Bernoulli polynomials of higher order:%
\begin{align}
\mathcal{E}_{n}^{\left(  \alpha+m\right)  }\left(  m;\lambda\right)   &
=\frac{\left(  -1\right)  ^{n}}{\lambda^{n}}\mathcal{E}_{n}^{\left(
\alpha+m\right)  }\left(  \alpha;\lambda^{-1}\right) \label{80}\\
&  =\frac{\left(  2\lambda^{-1}\right)  ^{m}}{m!}\binom{\alpha+m-1}{m}%
^{-1}\sum_{k=0}^{m}\left(  -1\right)  ^{k}%
\genfrac{[}{]}{0pt}{}{m}{k}%
\mathcal{E}_{n+k}^{\left(  \alpha\right)  }\left(  \lambda\right)
,\nonumber\\
\mathcal{B}_{n+m+l}^{\left(  m+l\right)  }\left(  m;\lambda\right)   &
=\frac{\left(  -1\right)  ^{n+m+l}}{\lambda^{n+m+l}}\mathcal{B}_{n+m+l}%
^{\left(  m+l\right)  }\left(  l;\lambda^{-1}\right) \label{31}\\
&  =\frac{\left(  l+m\right)  }{l\left(  -\lambda\right)  ^{m}}\binom
{n+m+l}{l}\sum_{k=0}^{m}\left(  -1\right)  ^{k}%
\genfrac{[}{]}{0pt}{}{m}{k}%
\binom{n+l+k}{l}^{-1}\mathcal{B}_{n+l+k}^{\left(  l\right)  }\left(
\lambda\right)  ,\nonumber
\end{align}
where $n,m\in%
\mathbb{N}
\cup\left\{  0\right\}  ,$ $l\in%
\mathbb{N}
,$ $\lambda\in%
\mathbb{C}
\backslash\left\{  0\right\}  ,$ $\alpha\in%
\mathbb{C}
$ such that $\operatorname{Re}\left(  \alpha\right)  >0$.
\end{theorem}

\begin{proof}
For $m=1$ in $\left(  \ref{19}\right)  $ we have
\begin{equation}
\mathcal{E}_{n}^{\left(  \alpha+1\right)  }\left(  1;\lambda\right)
=2\lambda^{-1}\binom{\alpha}{1}^{-1}\mathcal{E}_{n+1}^{\left(  \alpha\right)
}\left(  \lambda\right)  =%
\genfrac{[}{]}{0pt}{}{1}{1}%
2\left(  -\lambda\right)  ^{n-1}\binom{\alpha}{1}^{-1}\mathcal{E}%
_{n+1}^{\left(  \alpha\right)  }\left(  \lambda\right)  . \label{28}%
\end{equation}
Using $\left(  \ref{28}\right)  $ for $m=2$ in $\left(  \ref{25}\right)  $ we
get
\begin{equation}
\mathcal{E}_{n}^{\left(  \alpha+2\right)  }\left(  2;\lambda\right)
=\frac{\left(  2\lambda^{-1}\right)  ^{2}}{2!}\binom{\alpha}{2}^{-1}\left\{
\genfrac{[}{]}{0pt}{}{2}{2}%
\mathcal{E}_{n+2}^{\left(  \alpha\right)  }\left(  \lambda\right)  -%
\genfrac{[}{]}{0pt}{}{2}{1}%
\mathcal{E}_{n+1}^{\left(  \alpha\right)  }\left(  \lambda\right)  \right\}  .
\label{29}%
\end{equation}
Apostol-Euler polynomials of higher order\textbf{\ }satisfy the recurrence
relation $\cite[Eq. (1.16)]{WANG}$%
\[
\frac{\alpha\lambda}{2}\mathcal{E}_{n}^{\left(  \alpha+1\right)  }\left(
x+1;\lambda\right)  =x\mathcal{E}_{n}^{\left(  \alpha\right)  }\left(
x;\lambda\right)  -\mathcal{E}_{n+1}^{\left(  \alpha\right)  }\left(
x;\lambda\right)  .
\]
Thereby, setting $x=m,$ and $\alpha+m$ for $\alpha$ we get%
\begin{equation}
\mathcal{E}_{n}^{\left(  \alpha+m+1\right)  }\left(  m+1;\lambda\right)
=\frac{2m}{\left(  \alpha+m\right)  \lambda}\mathcal{E}_{n}^{\left(
\alpha+m\right)  }\left(  m;\lambda\right)  -\frac{2}{\left(  \alpha+m\right)
\lambda}\mathcal{E}_{n+1}^{\left(  \alpha+m\right)  }\left(  m;\lambda\right)
. \label{32}%
\end{equation}
Multiplying both sides of $\left(  \ref{80}\right)  $ by $\frac{2m}{\left(
\alpha+m\right)  \lambda}$ we get%
\begin{equation}
\frac{2m}{\left(  \alpha+m\right)  \lambda}\mathcal{E}_{n}^{\left(
\alpha+m\right)  }\left(  m;\lambda\right)  =\frac{\left(  2\lambda
^{-1}\right)  ^{m+1}}{\left(  m+1\right)  !}\binom{\alpha+m}{m+1}^{-1}%
\sum_{k=0}^{m}\left(  -1\right)  ^{k}m%
\genfrac{[}{]}{0pt}{}{m}{k}%
\mathcal{E}_{n+k}^{\left(  \alpha\right)  }\left(  \lambda\right)  .
\label{33}%
\end{equation}
Moreover taking $n+1$ for $n$ in $\left(  \ref{80}\right)  $ and multiplying
both sides by $\frac{2}{\left(  \alpha+m\right)  \lambda}$ we have%
\begin{equation}
\frac{2}{\left(  \alpha+m\right)  \lambda}\mathcal{E}_{n+1}^{\left(
\alpha+m\right)  }\left(  m;\lambda\right)  =\frac{\left(  2\lambda
^{-1}\right)  ^{m+1}}{\left(  m+1\right)  !}\binom{\alpha+m}{m+1}^{-1}%
\sum_{k=1}^{m+1}\left(  -1\right)  ^{k}%
\genfrac{[}{]}{0pt}{}{m}{k-1}%
\mathcal{E}_{n+k}^{\left(  \alpha\right)  }\left(  \lambda\right)  .
\label{34}%
\end{equation}
Subtracting $\left(  \ref{34}\right)  $ from $\left(  \ref{33}\right)  $,
using the well-known recurrence relation for Stirling number of the first
kind
\begin{equation}%
\genfrac{[}{]}{0pt}{}{m+1}{k}%
=m%
\genfrac{[}{]}{0pt}{}{m}{k}%
+%
\genfrac{[}{]}{0pt}{}{m}{k-1}
\label{35}%
\end{equation}
and the identity $\cite{LUO2}$%
\[
\mathcal{E}_{n}^{\left(  \alpha\right)  }\left(  \alpha-x;\lambda\right)
=\frac{\left(  -1\right)  ^{n}}{\lambda^{n}}\mathcal{E}_{n}^{\left(
\alpha\right)  }\left(  x;\lambda^{-1}\right)  ,
\]
we have the statement is true for $m+1.$

Equation $\left(  \ref{31}\right)  $ can be proved by the same method by using
the following recurrence relation for Apostol-Bernoulli polynomials of higher
order\textbf{\ }$\cite[Eq. (62)]{SRIandLUO2}$%
\[
\alpha\lambda\mathcal{B}_{n}^{\left(  \alpha+1\right)  }\left(  x+1;\lambda
\right)  =nx\mathcal{B}_{n-1}^{\left(  \alpha\right)  }\left(  x;\lambda
\right)  +\left(  \alpha-n\right)  \mathcal{B}_{n}^{\left(  \alpha\right)
}\left(  x;\lambda\right)  .
\]

\end{proof}

Now we mention some special cases of Theorem \ref{teo3} for Apostol-Euler and
Euler polynomials of higher order. For $\lambda=1$ we have a recurrence
relation for Euler polynomials of higher order involving Euler numbers of
higher order as in the following corollary.

\begin{corollary}
We have%
\begin{equation}
E_{n}^{\left(  \alpha+m\right)  }\left(  m\right)  =\left(  -1\right)
^{n}E_{n}^{\left(  \alpha+m\right)  }\left(  \alpha\right)  =\frac{\left(
2\right)  ^{m}}{m!\binom{\alpha+m-1}{m}}\sum_{k=0}^{m}\left(  -1\right)  ^{k}%
\genfrac{[}{]}{0pt}{}{m}{k}%
E_{n+k}^{\left(  \alpha\right)  }. \label{81}%
\end{equation}

\end{corollary}

Setting $\alpha=1$ and $\alpha=\lambda=1$ in $\left(  \ref{80}\right)  $ we
have recurrence relations for Apostol-Euler polynomials of higher order and
Euler polynomials of higher order involving Apostol-Euler and classical Euler numbers%

\begin{align}
\mathcal{E}_{n}^{\left(  m+1\right)  }\left(  m;\lambda\right)   &
=\frac{\left(  -1\right)  ^{n}}{\lambda^{n}}\mathcal{E}_{n}^{\left(
m+1\right)  }\left(  1;\lambda^{-1}\right)  =\frac{\left(  2\lambda
^{-1}\right)  ^{m}}{m!}\sum_{k=0}^{m}\left(  -1\right)  ^{k}%
\genfrac{[}{]}{0pt}{}{m}{k}%
\mathcal{E}_{n+k}\left(  \lambda\right)  ,\label{82}\\
E_{n}^{\left(  m+1\right)  }\left(  m\right)   &  =\left(  -1\right)
^{n}E_{n}^{\left(  m+1\right)  }\left(  1\right)  =\frac{\left(  2\right)
^{m}}{m!}\sum_{k=0}^{m}\left(  -1\right)  ^{k}%
\genfrac{[}{]}{0pt}{}{m}{k}%
E_{n+k}, \label{83}%
\end{align}
respectively. We also obtain closed form for some finite summations from
Theorem \ref{teo3}. Setting $\mathcal{E}_{0}^{\left(  \alpha\right)  }\left(
x;\lambda\right)  =2^{\alpha}\left(  \lambda+1\right)  ^{-\alpha}$ in $\left(
\ref{80}\right)  $ we have%
\begin{equation}
\sum_{k=0}^{m}\left(  -1\right)  ^{k}%
\genfrac{[}{]}{0pt}{}{m}{k}%
\mathcal{E}_{k}^{\left(  \alpha\right)  }\left(  \lambda\right)
=\frac{2^{\alpha}\lambda^{m}m!}{\left(  \lambda+1\right)  ^{\alpha+m}}%
\binom{\alpha+m-1}{m}. \label{60}%
\end{equation}
Furthermore for $\lambda=1,$ $\alpha=1$ and $\lambda=\alpha=1$ $\left(
\ref{60}\right)  $ we obtain
\begin{align}
\sum_{k=0}^{m}\left(  -1\right)  ^{k}%
\genfrac{[}{]}{0pt}{}{m}{k}%
E_{k}^{\left(  \alpha\right)  }  &  =\frac{m!}{2^{m}}\binom{\alpha+m-1}%
{m},\label{61}\\
\sum_{k=0}^{m}\left(  -1\right)  ^{k}%
\genfrac{[}{]}{0pt}{}{m}{k}%
\mathcal{E}_{k}\left(  \lambda\right)   &  =\frac{2\lambda^{m}m!}{\left(
\lambda+1\right)  ^{m+1}},\label{62}\\
\sum_{k=0}^{m}\left(  -1\right)  ^{k}%
\genfrac{[}{]}{0pt}{}{m}{k}%
E_{k}  &  =\frac{m!}{2^{m}}, \label{63}%
\end{align}
respectively. It is good note that the above finite summations can also be
obtained by applying Stirling transform to the identities in Remark \ref{re1}.

Now, we mention some special cases of Theorem \ref{teo3} for Apostol-Bernoulli
and Bernoulli polynomials of higher order. For $\lambda=1$ we have the
following recurrence relation for Bernoulli polynomials of higher order
involving Bernoulli numbers of higher order.

\begin{corollary}
We have
\begin{align}
B_{n+m+l}^{\left(  m+l\right)  }\left(  m\right)   &  =\left(  -1\right)
^{n+m+l}B_{n+m+l}^{\left(  m+l\right)  }\left(  l\right) \nonumber\\
&  =\frac{\left(  l+m\right)  }{l}\binom{n+m+l}{l}\sum_{k=0}^{m}\left(
-1\right)  ^{k}%
\genfrac{[}{]}{0pt}{}{m}{k}%
\binom{n+l+k}{l}^{-1}B_{n+l+k}^{\left(  l\right)  }. \label{90}%
\end{align}

\end{corollary}

Setting $l=1$ and $\lambda=l=1$ in $\left(  \ref{31}\right)  $ we have
following recurrence relations for Apostol-Bernoulli and Bernoulli polynomials
of higher order involving Apostol-Bernoulli and classical Bernoulli numbers
\begin{align}
\mathcal{B}_{n+m+1}^{\left(  m+1\right)  }\left(  m;\lambda\right)   &
=\frac{\left(  -1\right)  ^{n+m+1}}{\lambda^{n+m+1}}\mathcal{B}_{n+m+1}%
^{\left(  m+1\right)  }\left(  1;\lambda^{-1}\right)  \nonumber\\
&  =\frac{\left(  m+1\right)  \left(  n+m+1\right)  }{\left(  -\lambda\right)
^{m}}\sum_{k=0}^{m}\left(  -1\right)  ^{k}%
\genfrac{[}{]}{0pt}{}{m}{k}%
\frac{\mathcal{B}_{n+k+1}\left(  \lambda\right)  }{n+k+1},\label{91}\\
B_{n+m+1}^{\left(  m+1\right)  }\left(  m\right)   &  =\left(  -1\right)
^{n+m+1}B_{n+m+1}^{\left(  m+1\right)  }\left(  1\right)  \nonumber\\
&  =\left(  m+1\right)  \left(  n+m+1\right)  \sum_{k=0}^{m}\left(  -1\right)
^{k}%
\genfrac{[}{]}{0pt}{}{m}{k}%
\frac{B_{n+k+1}}{n+k+1},\label{92}%
\end{align}
respectively. Moreover for $n=0$ setting $\mathcal{B}_{m+l}^{\left(
m+l\right)  }\left(  x;\lambda\right)  =\left(  m+l\right)  !\left(
\lambda-1\right)  ^{-m-l}$ in $\left(  \ref{31}\right)  $ we have the
following finite summation
\begin{equation}
\sum_{k=0}^{m}\left(  -1\right)  ^{k}%
\genfrac{[}{]}{0pt}{}{m}{k}%
\binom{l+k}{l}^{-1}\mathcal{B}_{l+k}^{\left(  l\right)  }\left(
\lambda\right)  =\frac{l\left(  m+l-1\right)  !\left(  -\lambda\right)  ^{m}%
}{\left(  \lambda-1\right)  ^{m+l}}\binom{m+l}{l}^{-1}.\label{64}%
\end{equation}
Thus for $l=1$ in $\left(  \ref{64}\right)  $ we have
\[
\sum_{k=0}^{m}\left(  -1\right)  ^{k}%
\genfrac{[}{]}{0pt}{}{m}{k}%
\frac{\mathcal{B}_{k+1}\left(  \lambda\right)  }{k+1}=\frac{m!\left(
-\lambda\right)  ^{m}}{\left(  m+1\right)  \left(  \lambda-1\right)  ^{m+1}}.
\]
The above finite summations can also be obtained by applying Stirling
transform to the identities in Remark \ref{re2}. Besides, for $n=0$ in
$\left(  \ref{90}\right)  $ we have
\begin{equation}
B_{m+l}^{\left(  m+l\right)  }\left(  m\right)  =\left(  -1\right)
^{m+l}B_{m+l}^{\left(  m+l\right)  }\left(  l\right)  =\frac{\left(
l+m\right)  }{l}\binom{m+l}{l}\sum_{k=0}^{m}\left(  -1\right)  ^{k}%
\genfrac{[}{]}{0pt}{}{m}{k}%
\binom{l+k}{l}^{-1}B_{l+k}^{\left(  l\right)  }.\label{74}%
\end{equation}
When we compare $\left(  \ref{39}\right)  $ with $\left(  \ref{74}\right)  $
we obtain
\[
B_{m+l}^{\left(  m+l\right)  }\left(  m\right)  =\left(  -1\right)
^{m+l}B_{m+l}^{\left(  m+l\right)  }\left(  l\right)  =\binom{m+l}{l}%
^{-1}m!B_{m+l}^{\left(  m+l\right)  }.
\]
Replacing $m+l$ with $n$ gives
\begin{equation}
B_{n}^{\left(  n\right)  }\left(  n-l\right)  =\left(  -1\right)  ^{n}%
B_{n}^{\left(  n\right)  }\left(  l\right)  =\binom{n}{l}^{-1}\left(
n-l\right)  !B_{n}^{\left(  n\right)  }.\label{75}%
\end{equation}
Furthermore for $l=1$ we have
\begin{equation}
B_{n}^{\left(  n\right)  }\left(  1\right)  =\frac{\left(  -1\right)
^{n}\left(  n-1\right)  !}{n}B_{n}^{\left(  n\right)  },\text{ \ \ }%
n\geq1.\label{93}%
\end{equation}
To the writer's knowledge, this relationship between $B_{n}^{\left(  n\right)
}$ and $B_{n}^{\left(  n\right)  }\left(  l\right)  $ has not been pointed out
before. Howard deal with relationship between $B_{n}^{\left(  n\right)  }$ and
$B_{n}^{\left(  n\right)  }\left(  1\right)  $ and get the relations \cite[Eq.
(2.17)]{HOWARD}%
\begin{equation}
B_{n}^{\left(  n\right)  }\left(  1\right)  =\frac{1}{1-n}B_{n}^{\left(
n-1\right)  }\text{ and }B_{n}^{\left(  n\right)  }\left(  1\right)
=n!c_{n},\label{94}%
\end{equation}
where $c_{n}$ is the Bernoulli numbers of the second kind defined by Jordan in
\cite{JORDAN}. Also by comparing $\left(  \ref{93}\right)  $ with $\left(
\ref{94}\right)  $ gives a connection as
\[
c_{n}=\frac{\left(  -1\right)  ^{n}}{n^{2}}B_{n}^{\left(  n\right)  }.
\]
Finally, for $l=1$ in $\left(  \ref{74}\right)  $ we get a recurrence relation
for Bernoulli polynomials of higher order involving classical Bernoulli
numbers as
\[
B_{m+1}^{\left(  m+1\right)  }\left(  m\right)  =\left(  -1\right)
^{m+1}B_{m+1}^{\left(  m+1\right)  }\left(  1\right)  =m\left(  m+1\right)
\sum_{k=0}^{m}\left(  -1\right)  ^{k}%
\genfrac{[}{]}{0pt}{}{m}{k}%
\frac{B_{k+1}}{k+1}.
\]

\end{document}